\documentclass[reqno,11pt]{amsart}
\pdfoutput1

\usepackage{amssymb,color,hyperref,mathrsfs,stmaryrd, tikz-cd}
\usepackage{amsmath}
\usepackage{mathtools}

\usepackage{xcolor}
\usepackage[curve,matrix,arrow]{xy}

\newtheorem{theorem}{Theorem}[section]
\newtheorem{lemma}[theorem]{Lemma}

	\author{Julian Kaspczyk}
	\title{On finite groups with given $IC\Phi$-subgroups}
	\address{Technische Universität Dresden, Institut für Algebra, 01069 Dresden, Germany}
	\email{julian.kaspczyk@gmail.com}

\begin{document}
		\maketitle
		
		\begin{abstract}
		A subgroup $H$ of a group $G$ is said to be an $IC\Phi$-subgroup of $G$ if $H \cap [H,G] \le \Phi(H)$. We analyze the structure of a finite group $G$ under the assumption that some given subgroups of $G$ are $IC\Phi$-subgroups of $G$. A new characterization of finite abelian groups and some new criteria for $2$-nilpotence and nilpotence of finite groups will be obtained. Moreover, we will obtain two criteria for a finite group to lie in a given solvably saturated formation containing the class of finite supersolvable groups. 
		\end{abstract}
		
		
		\section{Introduction}
		All groups in this paper are implicitly assumed to be finite. Our notation and terminology are standard. The reader is referred to \cite{Gorenstein, Huppert, KurzweilStellmacher} for unfamiliar definitions on groups and to \cite{Guo2000} for unfamiliar definitions on classes of groups. 
		
		Given a group $G$ and a subgroup $H$ of $G$, we say that $H$ is an \textit{$IC\Phi$-subgroup} of $G$ provided that $H \cap [H,G] \le \Phi(H)$. This concept was introduced by Gao and Li in \cite{GaoLi2021} and further investigated by the author in \cite{Kaspczyk2021}. The papers \cite{GaoLi2021} and \cite{Kaspczyk2021} contain results that constrain the structure of a group $G$ under the condition that some given subgroups of $G$ are $IC\Phi$-subgroups of $G$. The following theorem is the main result of \cite{Kaspczyk2021}.  
		
		\begin{theorem}
			\label{theorem_first_paper}
			(\cite[Theorem 1.3]{Kaspczyk2021}) Let $p$ be a prime dividing the order of a group $G$, and let $P$ be a Sylow $p$-subgroup of $G$. Suppose that there is a subgroup $D$ of $P$ with $1 < |D| \le |P|$ such that any subgroup of $P$ with order $|D|$ is an $IC\Phi$-subgroup of $G$. If $|D| = 2$ and $|P| \ge 8$, assume moreover that any cyclic subgroup of $P$ with order $4$ is an $IC\Phi$-subgroup of $G$. Then $G$ is $p$-nilpotent.
		\end{theorem}
	
	The goal of the present paper is to further study how the structure of a group is influenced by its $IC\Phi$-subgroups. Our results show, together with \cite{GaoLi2021} and \cite{Kaspczyk2021}, that one often gets very rich information about a group when some of its subgroups are assumed to be $IC\Phi$-subgroups. 
	
	\subsection*{Groups some of whose primary subgroups are $IC\Phi$-subgroups}
		Let $Q_8$ denote the quaternion group with order $8$. Recall that a group $G$ is said to be \textit{$Q_8$-free} if no section of $G$ is isomorphic to $Q_8$. Our first main result is the following improvement of Theorem \ref{theorem_first_paper} and \cite[Theorem 3.1]{GaoLi2021}.
		
		\begin{theorem}
			\label{theorem_Q8_free_groups} 
			Let $G$ be a $Q_8$-free group such that any subgroup of $G$ with order $2$ is an $IC\Phi$-subgroup of $G$. Then $G$ is $2$-nilpotent. 
		\end{theorem}
	
	The condition in Theorem \ref{theorem_Q8_free_groups} that $G$ is $Q_8$-free is really necessary. For example, $Z(SL_2(3))$ is the only subgroup of $SL_2(3)$ with order $2$, and $Z(SL_2(3))$ is an $IC\Phi$-subgroup of $SL_2(3)$. But $SL_2(3)$ is not $2$-nilpotent. 
	
	Our second main result is a generalization of the following result of Gao and Li. 
	
	\begin{theorem}
	\label{GaoLi_theorem_3_5}
	(\cite[Theorem 3.5]{GaoLi2021}) Let $G$ be a group, and let $E$ be a normal subgroup of $G$ such that $G/E$ is supersolvable. If every maximal subgroup of every Sylow subgroup of $E$ is an $IC\Phi$-subgroup of $G$, then $G$ is supersolvable. 	
	\end{theorem} 

To state our generalization of Theorem \ref{GaoLi_theorem_3_5}, we recall some definitions. A class of groups $\mathfrak{F}$ is said to be a \textit{formation} if $\mathfrak{F}$ is closed under taking homomorphic images and subdirect products. A formation $\mathfrak{F}$ is said to be \textit{saturated} if whenever $G$ is a group with $G/\Phi(G) \in \mathfrak{F}$, we have $G \in \mathfrak{F}$. We say that a formation $\mathfrak{F}$ is \textit{solvably saturated} if whenever $G$ is a group and $N$ is a solvable normal subgroup of $G$ with $G/\Phi(N) \in \mathfrak{F}$, we have $G \in \mathfrak{F}$. Note that every saturated formation is solvably saturated. 

The class of all supersolvable groups is denoted by $\mathfrak{U}$. It is well-known that $\mathfrak{U}$ is a saturated and hence a solvably saturated formation. 

With these definitions at hand, we can now state our second main result. 

\begin{theorem}
	\label{theorem_formations_1} 
	Let $\mathfrak{F}$ be a solvably saturated formation containing $\mathfrak{U}$, let $G$ be a group, and let $E$ be a non-trivial normal subgroup of $G$ such that $G/E \in \mathfrak{F}$. Let $t := |\pi(E)|$, and let $p_1 < \dots < p_t$ be the distinct prime divisors of $|E|$. For each $1 \le i \le t$, let $P_i$ be a Sylow $p_i$-subgroup of $E$. Suppose that, for each $1 \le i \le t$, either $P_i$ is cyclic or there is a subgroup $D_i$ of $P_i$ with $1 < |D_i| \le |P_i|$ such that any subgroup of $P_i$ with order $|D_i|$ is an $IC\Phi$-subgroup of $G$. If $p_1 = 2$, $|D_1| = 2$ and $P_1$ is not $Q_8$-free, assume moreover that any cyclic subgroup of $P_1$ with order $4$ is an $IC\Phi$-subgroup of $G$. Then $G \in \mathfrak{F}$.   
\end{theorem}

Theorem \ref{GaoLi_theorem_3_5} is covered by Theorem \ref{theorem_formations_1}. Also, the proof of Theorem \ref{theorem_formations_1} given here is shorter than the proof of Theorem \ref{GaoLi_theorem_3_5} given in \cite{GaoLi2021}. 

Our third main result shows that Theorem \ref{theorem_formations_1} remains true when we replace the assumption that the subgroups $P_1, \dots, P_t$ are Sylow subgroups of $E$ by the assumption that they are Sylow subgroups of the generalized Fitting subgroup $F^{*}(E)$ of $E$. 

\begin{theorem}
	\label{theorem_formations_2} 
	Let $\mathfrak{F}$ be a solvably saturated formation containing $\mathfrak{U}$, let $G$ be a group, and let $E$ be a non-trivial normal subgroup of $G$ such that $G/E \in \mathfrak{F}$. Let $t := |\pi(F^{*}(E))|$, and let $p_1 < \dots < p_t$ be the distinct prime divisors of $|F^{*}(E)|$. For each $1 \le i \le t$, let $P_i$ be a Sylow $p_i$-subgroup of $F^{*}(E)$. Suppose that, for each $1 \le i \le t$, either $P_i$ is cyclic or there is a subgroup $D_i$ of $P_i$ with $1 < |D_i| \le |P_i|$ such that any subgroup of $P_i$ with order $|D_i|$ is an $IC\Phi$-subgroup of $G$. If $p_1 = 2$, $|D_1| = 2$ and $P_1$ is not $Q_8$-free, assume moreover that any cyclic subgroup of $P_1$ with order $4$ is an $IC\Phi$-subgroup of $G$. Then $G \in \mathfrak{F}$.   
\end{theorem}

All the above theorems are concerned with groups $G$ such that \textit{some} primary subgroups of $G$ are $IC\Phi$-subgroups of $G$. It is natural to ask what we can say about the structure of a group $G$ when \textit{every} primary subgroup of $G$ is an $IC\Phi$-subgroup of $G$. Clearly, any abelian group has this property. Also, one can check that any subgroup of $Q_8$ is an $IC\Phi$-subgroup of $Q_8$. Our fourth main result characterizes the abelian groups as the $Q_8$-free groups all of whose primary subgroups are $IC\Phi$-subgroups. 
		
\begin{theorem}
	\label{characterization_abelian_groups}
	Let $G$ be a group. Then the following are equivalent:
	\begin{enumerate}
		\item[(1)] $G$ is abelian. 
		\item[(2)] $G$ is $Q_8$-free, and any subgroup of $G$ is an $IC\Phi$-subgroup of $G$. 
		\item[(3)] $G$ is $Q_8$-free, and any primary subgroup of $G$ is an $IC\Phi$-subgroup of $G$.
	\end{enumerate}	
\end{theorem}

\subsection*{Groups all of whose maximal, $2$-maximal or $3$-maximal subgroups are $IC\Phi$-subgroups}
Let $n$ be a positive integer, and let $G$ be a group. A subgroup $H$ of $G$ is said to be \textit{$n$-maximal} in $G$ if there is a chain of subgroups $H = H_0 < H_1 < \dots < H_n = G$, where $H_i$ is maximal in $H_{i+1}$ for all $0 \le i \le n-1$. 

There are many results in finite group theory that describe the structure of a group $G$ under the assumption that, for a given positive integer $n$, all $n$-maximal subgroups of $G$ satisfy a given property.

Perhaps the most well-known result of this kind is due to Wielandt, who proved that a group $G$ is nilpotent if every maximal subgroup of $G$ is normal in $G$ (see \cite[Kapitel III, Hauptsatz 2.3]{Huppert}). Huppert proved that a group $G$ is supersolvable if every $2$-maximal subgroup of $G$ is normal in $G$ (see \cite[Satz 23]{Huppert1957}) or if $|G|$ is divisible by at least three primes and every $3$-maximal subgroup of $G$ is normal in $G$ (see \cite[Satz 24]{Huppert1957}). Janko proved that a solvable group $G$ is supersolvable if every $4$-maximal subgroup of $G$ is normal in $G$ and $|G|$ is divisible by at least four primes (see \cite[Theorem 3]{Janko}). Huppert's and Janko's results were strengthened by Asaad \cite{Asaad1989}. 

Mann \cite{Mann} proved a number of structural results about groups whose $n$-maximal subgroups, for some positive integer $n$, are subnormal. In the last decade, a number of results have been obtained on groups whose $n$-maximal subgroups, for some positive integer $n$, satisfy certain properties generalizing subnormality, see for example \cite{KovalevaSkiba, MonakhovKniahina, Monakhov, MonakhovKonovalova} (some of these results only deal with the case $n = 2$).

Other recent results on $n$-maximal subgroups and their influence on the structure of groups were obtained for example in \cite{Gaoetal2020, Gaoetal2021, Qian}.

As a development of the research on $n$-maximal subgroups, we will prove the following three theorems.
	
	\begin{theorem}
		\label{theorem_maximal_subgroups}
		Let $G$ be a group such that any maximal subgroup of $G$ is an $IC\Phi$-subgroup of $G$. Then $G$ is nilpotent.
	\end{theorem}

\begin{theorem}
	\label{theorem_second_maximal_subgroups}
	Let $G$ be a group. Suppose that $G$ has a non-trivial $2$-maximal subgroup and that any $2$-maximal subgroup of $G$ is an $IC\Phi$-subgroup of $G$. Then $G$ is nilpotent. 
\end{theorem}

\begin{theorem}
	\label{theorem_third_maximal_subgroups}
	Let $G$ be a group. Suppose that $G$ has a non-trivial $3$-maximal subgroup and that any $3$-maximal subgroup of $G$ is an $IC\Phi$-subgroup of $G$. Then either $G$ is nilpotent or $G \cong SL_2(3)$. 
\end{theorem}
		
		\section{Preliminaries} 
		In this section, we collect some results needed for the proofs of our main results. 
		
		\begin{lemma}
			\label{2.1}
			(\cite[Lemma 2.1]{GaoLi2021}) Let $G$ be a group, $H$ be an $IC\Phi$-subgroup of $G$, and $N$ be a normal subgroup of $G$. Then the following hold:
			\begin{enumerate}
				\item[(1)] If $H \le K \le G$, then $H$ is an $IC\Phi$-subgroup of $K$. 
				\item[(2)] If $N \le H$, then $H/N$ is an $IC\Phi$-subgroup of $G/N$. 
				\item[(3)] If $H$ is a $p$-group for some prime divisor $p$ of $\vert G \vert$ and $N$ is a $p'$-group, then $HN/N$ is an $IC\Phi$-subgroup of $G/N$. 
			\end{enumerate} 
		\end{lemma}
	
	\begin{lemma}
		\label{ICPhi_not_simple} 
		Let $G$ be a group possessing a proper non-trivial $IC\Phi$-subgroup $H$. Then $G$ is not simple. 
	\end{lemma}
	
	\begin{proof}
		Since $H$ is an $IC\Phi$-subgroup of $G$, we have $H \cap [H,G] \le \Phi(H)$. If $G = [H,G]$, then it follows that $H \le \Phi(H)$, which is impossible. Therefore, $[H,G]$ is a proper subgroup of $G$. Also, $[H,G]$ is normal in $G$. If $[H,G] \ne 1$, it follows that $G$ is not simple. If $[H,G] = 1$, then $H \le Z(G)$, and again it follows that $G$ is not simple.
	\end{proof}

\begin{lemma}
	\label{ICPhi_nilpotent_2}
	Let $G$ be a group, and let $H$ be an $IC\Phi$-subgroup of $G$. Suppose that $G' \le H$. Then $G$ is nilpotent. 
\end{lemma}

\begin{proof}
	We have $[G',G] \le H \cap [H,G] \le \Phi(H)$. Applying \cite[Kapitel III, Hilfssatz 3.3]{Huppert}, we conclude that $[G',G] \le \Phi(G)$. It follows that 
	\begin{equation*}
		[(G/\Phi(G))',G/\Phi(G)] = [G'\Phi(G)/\Phi(G),G/\Phi(G)] = [G',G]\Phi(G)/\Phi(G) = 1. 
	\end{equation*} 
So the lower central series of $G/\Phi(G)$ terminates at $1$. Consequently, $G/\Phi(G)$ is nilpotent, and \cite[Kapitel III, Satz 3.7]{Huppert} implies that $G$ is nilpotent. 
\end{proof}

\begin{lemma}
	\label{minimal_non_nilpotent_groups}
	(\cite[Theorem 3.4.11]{Guo2000}, \cite[Kapitel III, Satz 5.2]{Huppert}) Let $G$ be a minimal non-nilpotent group. Then: 
	\begin{enumerate}
		\item[(1)] $|G| = p^aq^b$ with distinct prime numbers $p$, $q$ and positive integers $a$, $b$, where $G$ has a normal Sylow $p$-subgroup $P$ and cyclic Sylow $q$-subgroups.
		\item[(2)] $P/\Phi(P)$ is a chief factor of $G$. 
		\item[(3)] If $P$ is abelian, then $P$ is elementary abelian.
	\end{enumerate} 
\end{lemma}

\begin{lemma}
	\label{Q8_free_minimal_non_2_nilpotent} 
Let $G$ be a $Q_8$-free minimal non-$2$-nilpotent group. Then $G$ has an elementary abelian Sylow $2$-subgroup. 
\end{lemma}

\begin{proof}
By \cite[Kapitel IV, Satz 5.4]{Huppert}, $G$ is minimal non-nilpotent. Lemma \ref{minimal_non_nilpotent_groups} (1) implies that $G$ has a normal Sylow $2$-subgroup $P$. 

Assume that $P$ is not elementary abelian. Then $P$ is non-abelian by Lemma \ref{minimal_non_nilpotent_groups} (3). Since $G$ is $Q_8$-free, we have that $P$ is $Q_8$-free. By a result of Ward, namely by \cite[Theorem 56.1]{BerkovichJanko}, any non-abelian $Q_8$-free $2$-group has a characteristic maximal subgroup. Since $P$ is normal in $G$, it follows that there is a maximal subgroup $P_1$ of $P$ which is normal in $G$. We have $\Phi(P) \le P_1$, and $P/\Phi(P)$ is a chief factor of $G$ by Lemma \ref{minimal_non_nilpotent_groups} (2). It follows that $\Phi(P) = P_1$. This implies that $P$ is cyclic and hence abelian. On the other hand, we have observed above that $P$ is non-abelian. This is a contradiction. 

So we have that $P$ is elementary abelian, and the lemma follows.   
\end{proof}

\begin{lemma}
	\label{Frobenius_p_nilpotence} 
	(\cite[Chapter 7, Theorem 4.5]{Gorenstein}) Let $p$ be a prime number, and let $G$ be a group. If $N_G(H)/C_G(H)$ is a $p$-group for any non-trivial $p$-subgroup $H$ of $G$, then $G$ is $p$-nilpotent. 
\end{lemma}

\begin{lemma}
	\label{Burnside_p_nilpotence}
	(\cite[7.2.2]{KurzweilStellmacher}) Let $G$ be a non-trivial group, and let $p$ be the smallest prime divisor of $|G|$. Suppose that the Sylow $p$-subgroups of $G$ are cyclic. Then $G$ is $p$-nilpotent. 
\end{lemma}

\begin{lemma}
\label{lemma_hypercenter} 
(\cite[Appendix C, Theorem 6.3]{Weinstein}) Let $p$ be a prime number, and let $P$ be a normal $p$-subgroup of a group $G$ such that $G/C_G(P)$ is a $p$-group. Then $P \le Z_{\infty}(G)$. 
\end{lemma}

To state the next two lemmas, we recall some definitions. Let $\mathfrak{F}$ be a formation, let $G$ be a group, and let $H/K$ be a chief factor of $G$. Then $H/K$ is said to be \textit{$\mathfrak{F}$-central} if $H/K \rtimes G/C_G(H/K) \in \mathfrak{F}$. Note that $H/K$ is $\mathfrak{U}$-central if and only if $H/K$ is cyclic. A normal subgroup $N$ of $G$ is said to be \textit{$\mathfrak{F}$-central} if any chief factor of $G$ below $N$ is $\mathfrak{F}$-central. The product of all $\mathfrak{F}$-central normal subgroups of $G$ is denoted by $Z_{\mathfrak{F}}(G)$.

\begin{lemma}
\label{lemma_formations_containing_U}
(\cite[Lemma 3.3]{GuoSkiba2012}) Let $\mathfrak{F}$ be a solvably saturated formation containing $\mathfrak{U}$. Let $G$ be a group and $E$ be a normal subgroup of $G$ such that $G/E \in \mathfrak{F}$ and $E \le Z_{\mathfrak{U}}(G)$. Then $G \in \mathfrak{F}$. 
\end{lemma}

\begin{lemma}
\label{lemma_formations_generalized_fitting} 
(\cite[Theorem B]{Skiba2010}) Let $\mathfrak{F}$ be a formation. Let $G$ be a group, and let $E$ be a normal subgroup of $G$ such that $F^{*}(E) \le Z_{\mathfrak{F}}(G)$. Then $E \le Z_{\mathfrak{F}}(G)$. 
\end{lemma}

\begin{lemma}
	\label{p_group_unique_subgroup_order_p}
	(\cite[5.3.7]{KurzweilStellmacher}) Let $p$ be a prime number, and let $P$ be a $p$-group such that $P$ has a unique subgroup with order $p$. Then either $P$ is cyclic, or $p = 2$ and $P$ is a generalized quaternion group. 
\end{lemma}

\begin{lemma}
	\label{lemma_second_maximal_subgroups} 
(\cite[Lemma 2.1]{Asaad1989}) Let $G$ be a group such that the trivial subgroup $1$ is a $2$-maximal subgroup of $G$ and such that $G$ has no non-trivial $2$-maximal subgroups. Then $|G| = pq$, where $p$ and $q$ are prime numbers (not necessarily distinct).
\end{lemma}

\begin{lemma}
\label{lemma_third_maximal_subgroups}
(\cite[Lemma 2.3]{Asaad1989}) Let $G$ be a group such that the trivial subgroup $1$ is a $3$-maximal subgroup of $G$ and such that $G$ has no non-trivial $3$-maximal subgroups. Then $|G|=pqr$, where $p$, $q$ and $r$ are prime numbers (not necessarily distinct).
\end{lemma}

\begin{lemma}
	\label{lemma_maximal_subgroups_solvable_groups} 
(\cite[Chapter 6, Theorem 1.5]{Gorenstein}) Let $G$ be a solvable group, and let $M$ be a maximal subgroup of $G$. Then $M$ has prime power index in $G$. 
\end{lemma}

\begin{lemma}
	\label{automorphisms_Q8}
(\cite[5.3.3]{KurzweilStellmacher}) $\mathrm{Aut}(Q_8)$ is isomorphic to the symmetric group $S_4$. 
\end{lemma}

\begin{lemma}
	\label{characterization_SL_2_3} 
(\cite[8.6.10]{KurzweilStellmacher}) Let $G$ be a $2$-closed group with order $24$ such that the Sylow $2$-subgroup of $G$ is isomorphic to $Q_8$. Then either $G \cong Q_8 \times C_3$ or $G \cong SL_2(3)$.  
\end{lemma}
		
\section{Proof of Theorem \ref{theorem_Q8_free_groups}} 
\begin{proof}[Proof of Theorem \ref{theorem_Q8_free_groups}]
Suppose that the theorem is false, and let $G$ be a minimal counterexample. 

Let $L$ be a proper subgroup of $G$. Then $L$ is $Q_8$-free since $G$ is $Q_8$-free. Also, by hypothesis, any subgroup of $L$ with order $2$ is an $IC\Phi$-subgroup of $G$. Lemma \ref{2.1} (1) implies that any subgroup of $L$ with order $2$ is an $IC\Phi$-subgroup of $L$. Consequently, $L$ satisfies the hypotheses of the theorem, and so $L$ is $2$-nilpotent by the minimality of $G$. It follows that $G$ is a minimal non-$2$-nilpotent group. 

Let $P \in \mathrm{Syl}_2(G)$. Then $P \ne 1$. By hypothesis, any subgroup of $P$ with order $2$ is an $IC\Phi$-subgroup of $G$. By Lemma \ref{Q8_free_minimal_non_2_nilpotent}, $P$ is elementary abelian. In particular, $P$ has no cyclic subgroup with order $4$. Applying Theorem \ref{theorem_first_paper} or \cite[Theorem 3.1]{GaoLi2021}, we conclude that $G$ is $2$-nilpotent. This contradiction completes the proof. 
\end{proof}

\section{Proofs of Theorems \ref{theorem_formations_1} and \ref{theorem_formations_2}} 

\begin{lemma}
\label{lemma_hypercentre}
Let $p$ be a prime number, let $G$ be a group, and let $P$ be a non-trivial normal $p$-subgroup of $G$. Suppose that there is a subgroup $D$ of $P$ with $1 < |D| \le |P|$ such that any subgroup of $P$ with order $|D|$ is an $IC\Phi$-subgroup of $G$. If $p = 2$, $|D| = 2$ and $P$ is not $Q_8$-free, assume moreover that any cyclic subgroup of $P$ with order $4$ is an $IC\Phi$-subgroup of $G$. Then $P \le Z_{\infty}(G)$. 
\end{lemma}

\begin{proof}
Let $q$ be a prime divisor of $|G|$ with $q \ne p$, and let $Q$ be a Sylow $q$-subgroup of $G$. Set $H := PQ$. We have $H \le G$ since $P$ is normal in $G$. Note that $P$ is a Sylow $p$-subgroup of $H$. By Lemma \ref{2.1} (1), any subgroup of $P$ with order $|D|$ is an $IC\Phi$-subgroup of $H$. Also, if $p = 2$, $|D| = 2$ and $P$ is not $Q_8$-free, we have that any cyclic subgroup of $P$ with order $4$ is an $IC\Phi$-subgroup of $H$. Theorems \ref{theorem_first_paper} and \ref{theorem_Q8_free_groups} imply that $H$ is $p$-nilpotent. This implies that $H = P \times Q$, and so we have $Q \le C_G(P)$. Since $q$ was arbitrarily chosen, it follows that $O^p(G) \le C_G(P)$. Hence, $G/C_G(P)$ is a $p$-group. Lemma \ref{lemma_hypercenter} implies that $P \le Z_{\infty}(G)$. 
\end{proof}

\begin{proof}[Proof of Theorem \ref{theorem_formations_1}]
Suppose that the theorem is false, and let $(G,E)$ be a counterexample such that $|G| + |E|$ is minimal. 

Set $p := p_1$. We claim that $E$ is $p$-nilpotent. This follows from Lemma \ref{Burnside_p_nilpotence} when $P_1$ is cyclic. Assume now that $P_1$ is not cyclic. Then $P_1$ has a subgroup $D_1$ with $1 < |D_1| \le |P_1|$ such that any subgroup of $P_1$ with order $|D_1|$ is an $IC\Phi$-subgroup of $G$. Also, if $p = 2$, $|D_1| = 2$ and $P_1$ is not $Q_8$-free, then any cyclic subgroup of $P_1$ with order $4$ is an $IC\Phi$-subgroup of $G$. Applying Lemma \ref{2.1} (1), Theorem \ref{theorem_first_paper} and Theorem \ref{theorem_Q8_free_groups}, we conclude that $E$ is $p$-nilpotent, as claimed. 

Assume that $O_{p'}(E) \ne 1$. From Lemma \ref{2.1} (3), we see that $(G/O_{p'}(E),E/O_{p'}(E))$ satisfies the hypotheses of the theorem. So we have $G/O_{p'}(E) \in \mathfrak{F}$ by the minimality of $(G,E)$. Therefore, $(G,O_{p'}(E))$ also satisfies the hypotheses of the theorem. The minimality of $(G,E)$ implies that $G \in \mathfrak{F}$. This is a contradiction, and so we have $O_{p'}(E) = 1$. 

We show now that $E \le Z_{\mathfrak{U}}(G)$. Since $E$ is $p$-nilpotent and since $O_{p'}(E) = 1$, we have that $E = P_1$. If $P_1$ is cyclic, then it follows that $E = P_1 \le Z_{\mathfrak{U}}(G)$. If $P_1$ is not cyclic, then the hypotheses of the theorem and Lemma \ref{lemma_hypercentre} imply that $E = P_1 \le Z_{\infty}(G)$ and thus $E \le Z_{\mathfrak{U}}(G)$. 

Now Lemma \ref{lemma_formations_containing_U} implies that $G \in \mathfrak{F}$. This contradiction completes the proof. 
\end{proof}

\begin{proof}[Proof of Theorem \ref{theorem_formations_2}]
Arguing as at the beginning of the proof of Theorem \ref{theorem_formations_1}, we see that $F^{*}(E)$ is $p_1$-nilpotent. With $N_1 := O_{(p_1)'}(F^{*}(E))$, we thus have $F^{*}(E)/N_1 \cong P_1$. 

Assume that $t > 1$. Then $P_2 \in \mathrm{Syl}_{p_2}(N_1)$, and again we can argue as at the beginning of the proof of Theorem \ref{theorem_formations_1} to see that $N_1$ is $p_2$-nilpotent. With $N_2 := O_{(p_2)'}(N_1)$, we thus have $N_1/N_2 \cong P_2$.  

Repeating this argumentation, we see that $G$ has a Sylow tower of supersolvable type, i.e. there is a chain $F^{*}(E) = N_0 > N_1 > \dots > N_t = 1$ of normal subgroups of $F^{*}(E)$ such that $N_{i-1}/N_i \cong P_i$ for all $1 \le i \le t$. It follows that $F^{*}(E)$ is solvable.

As is well-known, $F^{*}(E)$ is generated by $F(E)$ together with the components of $E$. Every component of $E$ is a non-solvable subgroup of $E$. Since $F^{*}(E)$ is solvable, it follows that $E$ does not possess any components. Consequently, we have $F^{*}(E) = F(E)$. 

Let $1 \le i \le t$. Since $F(E)$ is nilpotent and $P_i \in \mathrm{Syl}_{p_i}(F(E))$, we have that $P_i$ is characteristic in $F(E)$. As $F(E) \trianglelefteq G$, it follows that $P_i \trianglelefteq G$. If $P_i$ is cyclic, then $P_i \le Z_{\mathfrak{U}}(G)$. If $P_i$ is not cyclic, then the hypotheses of the theorem and Lemma \ref{lemma_hypercentre} imply that $P_i \le Z_{\infty}(G)$ and hence $P_i \le Z_{\mathfrak{U}}(G)$. Since $i$ was arbitrarily chosen, it follows that $F(E) \le Z_{\mathfrak{U}}(G)$. 

Applying Lemmas \ref{lemma_formations_generalized_fitting} and \ref{lemma_formations_containing_U}, we conclude that $G \in \mathfrak{F}$. 
\end{proof} 

\section*{Proof of Theorem \ref{characterization_abelian_groups}}
		\begin{lemma}
			\label{lemma_characterization_abelian_groups}
			Let $p$ be a prime number, and let $P$ be a $p$-group such that any subgroup of $P$ is an $IC\Phi$-subgroup of $P$. If $p = 2$, assume moreover that $P$ is $Q_8$-free. Then $P$ is abelian. 
		\end{lemma}
	
	\begin{proof}
		Suppose that the lemma is false, and let $P$ be a minimal counterexample. We will derive a contradiction in three steps. 
		
		\medskip
		
		(1) $P'$ is minimal normal in $P$, and there are no minimal normal subgroups of $P$ other than $P'$.
		
		Clearly $P \ne 1$, and so $P$ has a minimal normal subgroup, say $N$. We show that $N = P'$. 
		
		Let $N \le H \le P$. By hypothesis, $H$ is an $IC\Phi$-subgroup of $P$. Lemma \ref{2.1} (2) shows that $H/N$ is an $IC\Phi$-subgroup of $P/N$. Since $H$ was arbitrarily chosen, it follows that any subgroup of $P/N$ is an $IC\Phi$-subgroup of $P/N$. Also, if $p = 2$, then $P/N$ is $Q_8$-free since $P$ is $Q_8$-free. Therefore, $P/N$ satisfies the hypotheses of the lemma, and so $P/N$ is abelian by the minimality of $P$. 
		
		It follows that $P' \le N$. Noticing that $P' \ne 1$ since $P$ is not abelian, we conclude that $N = P'$, as required.
		
		\medskip
		
		(2) We have $|P'| = p$, and there is no subgroup of $P$ with order $p$ other than $P'$.

		We have $|P'| = p$ since $P'$ is minimal normal in $P$. Assume that there is a subgroup $Q$ of $P$ with $|Q| = p$ and $Q \ne P'$. Set $H := P'Q \le P$. Note that $\Phi(H) = 1$.  
		
		By (1), $Q$ is not normal in $P$. So we have $Q \not\le Z(P)$ and hence $H \not\le Z(P)$. Thus $[H,P] \ne 1$. Clearly $[H,P] \trianglelefteq P$ and $[H,P] \le P'$. So we have $[H,P] = P'$ by (1). 
		
		By hypothesis, $H$ is an $IC\Phi$-subgroup of $P$. It follows that $P' = H \cap P' = H \cap [H,P] \le \Phi(H) = 1$. This is a contradiction, and so there is no subgroup of $P$ with order $p$ other than $P'$. 
		
		\medskip
		
		(3) The final contradiction. 
		
		By (2), $P$ has precisely one subgroup with order $p$. Moreover, $P$ cannot be a generalized quaternion group since $P$ is $Q_8$-free by hypothesis. Lemma \ref{p_group_unique_subgroup_order_p} implies that $P$ is cyclic and hence abelian. This final contradiction completes the proof. 
		\end{proof}
	
	\begin{proof}[Proof of Theorem \ref{characterization_abelian_groups}]
		(1) $\Rightarrow$ (2): Suppose that $G$ is abelian. Then any section of $G$ is abelian. In particular, $G$ is $Q_8$-free. Also, if $H$ is a subgroup of $G$, then $H \cap [H,G] = H \cap 1 = 1 \le \Phi(H)$, so that $H$ is an $IC\Phi$-subgroup of $G$. Thus (2) holds. 
		
		(2) $\Rightarrow$ (3): Clear. 
		
		(3) $\Rightarrow$ (1): Suppose that $G$ is $Q_8$-free and that any primary subgroup of $G$ is an $IC\Phi$-subgroup of $G$. If $G = 1$, then there is nothing to show. Thus we assume that $G \ne 1$. Set $t := |\pi(G)|$, and let $p_1,\dots,p_t$ be the distinct prime divisors of $|G|$. For each $1 \le i \le t$, let $P_i$ be a Sylow $p_i$-subgroup of $G$. 
		
		Let $1 \le i \le t$. Since any primary subgroup of $G$ is an $IC\Phi$-subgroup of $G$, we have that $P_i$ is an $IC\Phi$-subgroup of $G$. Theorem \ref{theorem_first_paper} implies that $G$ is $p_i$-nilpotent. Since $i$ was arbitrarily chosen, we have that $G$ is $p$-nilpotent for any prime divisor $p$ of $|G|$. Consequently, $G$ is nilpotent, and so we have $G = P_1 \times \dots \times P_t$. 
		
		Let $1 \le i \le t$. Then any subgroup of $P_i$ is an $IC\Phi$-subgroup of $G$. Lemma \ref{2.1} (1) implies that any subgroup of $P_i$ is an $IC\Phi$-subgroup of $P_i$. Also, $P_i$ is $Q_8$-free since $G$ is $Q_8$-free. Lemma \ref{lemma_characterization_abelian_groups} implies that $P_i$ is abelian. 
		
		Consequently, $G$ is a direct product of abelian groups, and so $G$ is abelian as well. 
	\end{proof}

\section*{Proofs of Theorems \ref{theorem_maximal_subgroups}, \ref{theorem_second_maximal_subgroups} and \ref{theorem_third_maximal_subgroups}}

\begin{proof}[Proof of Theorem \ref{theorem_maximal_subgroups}]
	Suppose that the theorem is false, and let $G$ be a minimal counterexample.
	
	Clearly, $G$ has a non-trivial maximal subgroup $M$. By hypothesis, $M$ is an $IC\Phi$-subgroup of $G$. Lemma \ref{ICPhi_not_simple} implies that $G$ is not simple. 
	
	Let $N$ be a proper non-trivial normal subgroup of $G$, and let $N \le M \le G$ such that $M/N$ is a maximal subgroup of $G/N$. Then $M$ is a maximal subgroup of $G$. So $M$ is an $IC\Phi$-subgroup of $G$. Lemma \ref{2.1} (2) implies that $M/N$ is an $IC\Phi$-subgroup of $G/N$. Since $M$ was arbitrarily chosen, it follows that any maximal subgroup of $G/N$ is an $IC\Phi$-subgroup of $G/N$. The minimality of $G$ implies that $G/N$ is nilpotent. 
	
	It follows that $(G/N)' = G'N/N$ is a proper subgroup of $G/N$. This implies that $G'$ is a proper subgroup of $G$. So there is a maximal subgroup $M$ of $G$ with $G' \le M$. By hypothesis, $M$ is an $IC\Phi$-subgroup of $G$. Lemma \ref{ICPhi_nilpotent_2} implies that $G$ is nilpotent. This is a contradiction to the choice of $G$, and so the proof is complete.  
\end{proof}

\begin{proof}[Proof of Theorem \ref{theorem_second_maximal_subgroups}]
Suppose that the theorem is false, and let $G$ be a (not necessarily minimal) counterexample. 

Let $M$ be a maximal subgroup of $G$. By hypothesis, every maximal subgroup of $M$ is an $IC\Phi$-subgroup of $G$. Lemma \ref{2.1} (1) implies that any maximal subgroup of $M$ is an $IC\Phi$-subgroup of $M$. So $M$ is nilpotent by Theorem \ref{theorem_maximal_subgroups}. Since $M$ was arbitrarily chosen, it follows that any maximal subgroup of $G$ is nilpotent. Consequently, $G$ is minimal non-nilpotent. 

By Lemma \ref{minimal_non_nilpotent_groups} (1), we have $|G| = p^aq^b$ with distinct prime numbers $p$, $q$ and positive integers $a$, $b$, where $G$ has a normal Sylow $p$-subgroup $P$ and cyclic Sylow $q$-subgroups. 

We claim that $b = 1$. Assume, for sake of contradiction, that $b \ge 2$. Clearly, $G/P$ is cyclic with order $q^b$. Let $P \le H < G$ such that $H/P$ is the unique $2$-maximal subgroup of $G/P$. Then $H$ is a $2$-maximal subgroup of $G$. By hypothesis, $H$ is an $IC\Phi$-subgroup of $G$, and we have $G' \le P \le H$. Lemma \ref{ICPhi_nilpotent_2} implies that $G$ is nilpotent. This contradiction shows that $b = 1$, as claimed.

It follows that any maximal subgroup of $P$ is a $2$-maximal subgroup of $G$. Consequently, any maximal subgroup of $P$ is an $IC\Phi$-subgroup of $G$. Also, we have $|P| > p$, since otherwise $G$ would not possess a non-trivial $2$-maximal subgroup. Applying Theorem \ref{theorem_first_paper}, we conclude that $G$ is $p$-nilpotent. Therefore, $G$ has a normal Sylow $q$-subgroup. Consequently, any Sylow subgroup of $G$ is normal in $G$. It follows that $G$ is nilpotent. This contradiction completes the proof. 
\end{proof}

We need the following lemma to prove Theorem \ref{theorem_third_maximal_subgroups}. 

\begin{lemma}
	\label{lemma_for_theorem_third_maximal_subgroups} 
Let $G$ be a group such that any $3$-maximal subgroup of $G$ is an $IC\Phi$-subgroup of $G$. Then $G$ is solvable. 
\end{lemma}

\begin{proof}
Suppose that the lemma is false, and let $G$ be a minimal counterexample. We will derive a contradiction in several steps. 

\medskip 

(1) $G$ has a non-trivial $2$-maximal subgroup. 

Clearly, $G$ has a non-trivial maximal subgroup $M$. Any maximal subgroup of $M$ is a $2$-maximal subgroup of $G$. Consequently, the set of $2$-maximal subgroups of $G$ is not empty. If $1$ is the only $2$-maximal subgroup of $G$, then $G$ is solvable as a consequence of Lemma \ref{lemma_second_maximal_subgroups}, a contradiction. So $G$ has a non-trivial $2$-maximal subgroup. 

\medskip

(2) $G$ has a non-trivial $3$-maximal subgroup. 

By (1), $G$ has a non-trivial $2$-maximal subgroup, say $H$. Any maximal subgroup of $H$ is a $3$-maximal subgroup of $G$. Consequently, the set of $3$-maximal subgroups of $G$ is not empty. If $1$ is the only $3$-maximal subgroup of $G$, then $G$ is solvable as a consequence of Lemma \ref{lemma_third_maximal_subgroups}, a contradiction. So $G$ has a non-trivial $3$-maximal subgroup. 

\medskip

(3) $G$ is not simple. 

By (2), $G$ has a non-trivial $3$-maximal subgroup. By hypothesis, any $3$-maximal subgroup of $G$ is an $IC\Phi$-subgroup of $G$. Consequently, $G$ has a proper non-trivial $IC\Phi$-subgroup. Lemma \ref{ICPhi_not_simple} implies that $G$ is not simple. 

\medskip

(4) Any proper subgroup of $G$ is solvable. 

Let $M$ be a maximal subgroup of $G$. It suffices to show that $M$ is solvable. If $M$ has no non-trivial maximal subgroup, then $M$ has prime order, and so $M$ is solvable. 

Suppose now that $M$ has a non-trivial maximal subgroup. Then the set of $2$-maximal subgroups of $M$ is not empty. If $1$ is the only $2$-maximal subgroup of $M$, then $M$ is solvable as a consequence of Lemma \ref{lemma_second_maximal_subgroups}. 

By hypothesis, any $2$-maximal subgroup of $M$ is an $IC\Phi$-subgroup of $G$. Lemma \ref{2.1} (1) implies that any $2$-maximal subgroup of $M$ is an $IC\Phi$-subgroup of $M$. So, if $M$ has a non-trivial $2$-maximal subgroup, then Theorem \ref{theorem_second_maximal_subgroups} implies that $M$ is nilpotent and hence solvable. 

\medskip 

(5) The final contradiction. 

By (3), $G$ has a proper non-trivial normal subgroup, say $N$. Let $N \le H \le G$ such that $H/N$ is a $3$-maximal subgroup of $G/N$. Then $H$ is a $3$-maximal subgroup of $G$. So, by hypothesis, $H$ is an $IC\Phi$-subgroup of $G$. Lemma \ref{2.1} (2) implies that $H/N$ is an $IC\Phi$-subgroup of $G/N$. Since $H$ was arbitrarily chosen, it follows that any $3$-maximal subgroup of $G/N$ is an $IC\Phi$-subgroup of $G/N$. The minimality of $G$ implies that $G/N$ is solvable. Also, $N$ is solvable by (4). It follows that $G$ is solvable. This final contradiction completes the proof. 
\end{proof}

\begin{proof}[Proof of Theorem \ref{theorem_third_maximal_subgroups}]
	Let $G$ be a non-nilpotent group such that $G$ has a non-trivial $3$-maximal subgroup and such that any $3$-maximal subgroup of $G$ is an $IC\Phi$-subgroup of $G$. Our task is to show that $G$ is isomorphic to $SL_2(3)$. We accomplish the proof step by step. 
	
	\medskip
	
	(1) If $M$ is a non-nilpotent maximal subgroup of $G$, then there exist distinct prime numbers $p$ and $q$ such that $|M| = pq$ and such that $|G:M|$ is a power of $p$. 
	
	Let $M$ be a non-nilpotent maximal subgroup of $G$. Then $M$ has a non-trivial maximal subgroup, and so the set of $2$-maximal subgroups of $M$ is not empty. By hypothesis, any $2$-maximal subgroup of $M$ is an $IC\Phi$-subgroup of $G$. Lemma \ref{2.1} (1) implies that any $2$-maximal subgroup of $M$ is an $IC\Phi$-subgroup of $M$. Since $M$ is not nilpotent, it follows from Theorem \ref{theorem_second_maximal_subgroups} that the trivial subgroup $1$ is the only $2$-maximal subgroup of $M$. Applying Lemma \ref{lemma_second_maximal_subgroups}, we conclude that $|M| = pq$, where $p$ and $q$ are prime numbers. We have $p \ne q$ since $M$ is not nilpotent. 
	
	By Lemma \ref{lemma_for_theorem_third_maximal_subgroups}, $G$ is solvable. So, by Lemma \ref{lemma_maximal_subgroups_solvable_groups}, the index $|G : M|$ is a power of a prime number $r$. Assume that $r \not\in \lbrace p,q \rbrace$. As a solvable group, $G$ has a Sylow system. Hence, there exist $P \in \mathrm{Syl}_p(G)$, $Q \in \mathrm{Syl}_q(G)$ and $R \in \mathrm{Syl}_r(G)$ such that $P$, $Q$ and $R$ are pairwise permutable. Since $|P| = p$ and $|Q| = q$, we have that $R$ is maximal in $RQ$ and that $RQ$ is maximal in $G$. Consequently, any maximal subgroup of $R$ is a $3$-maximal subgroup of $G$. Therefore, any maximal subgroup of $R$ is an $IC\Phi$-subgroup of $G$. We have $|R| > r$, since otherwise $G$ would not possess a non-trivial $3$-maximal subgroup. Applying Theorem \ref{theorem_first_paper}, we conclude that $G$ is $r$-nilpotent. This implies that $M = O_{r'}(G) \trianglelefteq G$. Since $|G:M| = |R| > r$, it follows that $M$ cannot be maximal in $G$. This contradiction shows that $r \in \lbrace p,q \rbrace$, and without loss of generality, we may assume that $r = p$.
	
	\medskip
	
	(2) $G$ is minimal non-nilpotent.
	
	Assume that $G$ is not minimal non-nilpotent. Then $G$ has a non-nilpotent maximal subgroup $M$. By (1), there exist distinct prime numbers $p$ and $q$ such that $|M| = pq$ and such that $|G:M|$ is a power of $p$.
	
	Let $P$ be a Sylow $p$-subgroup of $G$, and let $Q$ be a Sylow $q$-subgroup of $G$. Then $|G:P| = q$, and so we have $|P| \ge p^3$, since otherwise $G$ would not possess a non-trivial $3$-maximal subgroup. 
	
	Assume that $Q \trianglelefteq G$. Then $M/Q$ is a minimal subgroup of $G/Q$, and we have $|G/Q| = |P| \ge p^3$. This is a contradiction to the maximality of $M$ in $G$. Consequently, $Q$ is not normal in $G$. 
	
	Any $2$-maximal subgroup of $P$ is a $3$-maximal subgroup of $G$. So, by hypothesis, any $2$-maximal subgroup of $P$ is an $IC\Phi$-subgroup of $G$. If $p$ is odd or if $p = 2$ and $P \not\cong Q_8$, then Theorems \ref{theorem_first_paper} and \ref{theorem_Q8_free_groups} imply that $G$ is $p$-nilpotent, which is impossible since $Q$ is not normal in $G$. So we have $p = 2$ and $P \cong Q_8$.
	
	Now let $N$ be a minimal normal subgroup of $G$. Since $G$ is solvable, $N$ is a primary subgroup of $G$. Also $N \not\le Q$ since $|Q| = q$ and $Q$ is not normal in $G$. So we have $N \le P$. 
	
	Since $P \cong Q_8$, we have that $Z(P)$ is the only subgroup of $P$ with order $2$. Hence $Z(P) \le N$, and the minimal normality of $N$ in $G$ implies that $N = Z(P)$. It follows that $N$ is the only subgroup of $G$ with order $2$. Consequently, $N \in \mathrm{Syl}_2(M)$, which easily implies that $M$ is nilpotent. This is a contradiction, and so (2) holds.

	\medskip
	
	(3) $G$ has a normal Sylow $2$-subgroup $P$. We have $P \cong Q_8$ and $|G:P| = 3$.  
	
	In order to prove this, we use similar arguments as in the proof of Theorem \ref{theorem_second_maximal_subgroups}. 
	
	By (2), $G$ is minimal non-nilpotent. So, by Lemma \ref{minimal_non_nilpotent_groups} (1), we have $|G| = p^aq^b$ with distinct prime numbers $p$, $q$ and positive integers $a$, $b$, where $G$ has a normal Sylow $p$-subgroup $P$ and cyclic Sylow $q$-subgroups.
	
	Assume that $b \ge 3$. Clearly, $G/P$ is cyclic with order $q^b$. Let $P \le H < G$ such that $H/P$ is the unique $3$-maximal subgroup of $G/P$. Then $H$ is a $3$-maximal subgroup of $G$. By hypothesis, $H$ is an $IC\Phi$-subgroup of $G$, and we have $G' \le P \le H$. Lemma \ref{ICPhi_nilpotent_2} implies that $G$ is nilpotent. This contradiction shows that $b \le 2$. 
	
	Assume that $b = 2$. Then any maximal subgroup of $P$ is $3$-maximal in $G$, and so we have that any maximal subgroup of $P$ is an $IC\Phi$-subgroup of $G$. We have $|P| > p$, since otherwise $G$ would not possess a non-trivial $3$-maximal subgroup. Theorem \ref{theorem_first_paper} implies that $G$ is $p$-nilpotent. But $G$ is also $q$-nilpotent, and so $G$ is nilpotent. This contradiction shows that $b = 1$. 
	
	It follows that any $2$-maximal subgroup of $P$ is $3$-maximal in $G$. Therefore, any $2$-maximal subgroup of $P$ is an $IC\Phi$-subgroup of $G$. We have $|P| \ge p^3$, since otherwise $G$ would not possess a non-trivial $3$-maximal subgroup. If $p$ is odd or if $p = 2$ and $P \not\cong Q_8$, then Theorems \ref{theorem_first_paper} and \ref{theorem_Q8_free_groups} imply that $G$ is $p$-nilpotent, which leads to a contradiction as above. So we have $p = 2$ and $P \cong Q_8$. 
	
	If $U$ is a non-trivial proper subgroup of $P$, then $U$ is a cyclic $2$-group, and this implies that $N_G(U)/C_G(U)$ is a $2$-group. Since $G$ is not $2$-nilpotent, Lemma \ref{Frobenius_p_nilpotence} implies that $G/C_G(P)$ is not a $2$-group. Since $G/C_G(P)$ is isomorphic to a subgroup of $\mathrm{Aut}(P)$, and since $\mathrm{Aut}(P)$ has order $24$ by Lemma \ref{automorphisms_Q8}, we conclude that $|G/C_G(P)|$ is divisible $3$. Therefore, $|G|$ is divisible by $3$. Since $P$ has prime index in $G$, it follows that $|G:P| = 3$. 
	
	\medskip
	
	(4) Conclusion. 
	
	Applying Lemma \ref{characterization_SL_2_3}, we deduce from (3) that $G \cong SL_2(3)$. So we have reached the desired conclusion. 
	\end{proof}

	\end{document}